\documentclass[11pt]{amsart}

\usepackage{epigamath}

\usepackage[english]{babel}

\numberwithin{equation}{section}

\usepackage{enumitem}
\setlist[enumerate,1]{label={\rm(\arabic*)}, ref={\rm\arabic*}}

\usepackage[utf8]{inputenc}
\usepackage[T1]{fontenc}

\usepackage{color}
\usepackage{xcolor} 
\usepackage{calligra}
\usepackage{mathrsfs}
\usepackage[all]{xy}

\newtheorem{theo}{Theorem}[section]
\newtheorem{prop}[theo]{Proposition}
\newtheorem{lemm}[theo]{Lemma}

\theoremstyle{definition}
\newtheorem{defi}[theo]{Definition}
\newtheorem{setup}[theo]{Setting}

\theoremstyle{remark}
\newtheorem{rem}[theo]{Remark}
\newtheorem{ex}[theo]{Example}

\newcommand{\Supp}[0]{\operatorname{Supp}}

\newcommand{\codim}[0]{\operatorname{codim}}

\newcommand{\ddbar}{dd^c}
\newcommand{\deldel}{\sqrt{-1}\partial \overline{\partial}}

\newcommand{\e}{\varepsilon}

\newcommand{\I}[1]{\mathcal{I}(#1)}

\DeclareMathOperator{\reg}{reg}
\DeclareMathOperator{\sing}{sing}

\newcommand{\cal}[1]{\mathcal{#1}}
\newcommand{\bb}[1]{\mathbb{#1}}

\makeatletter
\newcommand*{\da@rightarrow}{\mathchar"0\hexnumber@\symAMSa 4B }
\newcommand*{\da@leftarrow}{\mathchar"0\hexnumber@\symAMSa 4C }
\newcommand*{\xdashrightarrow}[2][]{%
  \mathrel{%
    \mathpalette{\da@xarrow{#1}{#2}{}\da@rightarrow{\,}{}}{}%
  }%
}
\newcommand{\xdashleftarrow}[2][]{%
  \mathrel{%
    \mathpalette{\da@xarrow{#1}{#2}\da@leftarrow{}{}{\,}}{}%
  }%
}
\newcommand*{\da@xarrow}[7]{%
  \sbox0{$\ifx#7\scriptstyle\scriptscriptstyle\else\scriptstyle\fi#5#1#6\m@th$}%
  \sbox2{$\ifx#7\scriptstyle\scriptscriptstyle\else\scriptstyle\fi#5#2#6\m@th$}%
  \sbox4{$#7\dabar@\m@th$}%
  \dimen@=\wd0 %
  \ifdim\wd2 >\dimen@
    \dimen@=\wd2 %   
  \fi
  \count@=2 %
  \def\da@bars{\dabar@\dabar@}%
  \@whiledim\count@\wd4<\dimen@\do{%
    \advance\count@\@ne
    \expandafter\def\expandafter\da@bars\expandafter{%
      \da@bars
      \dabar@ 
    }%
  }%  
  \mathrel{#3}%
  \mathrel{%   
    \mathop{\da@bars}\limits
    \ifx\\#1\\%
    \else
      _{\copy0}%
    \fi
    \ifx\\#2\\%
    \else
      ^{\copy2}%
    \fi
  }%   
  \mathrel{#4}%
}
\makeatother

\DeclareMathOperator{\id}{id}
\DeclareMathOperator{\Bs}{Bs}
\DeclareMathOperator{\tor}{tor}
\DeclareMathOperator{\Bl}{Bl}
\DeclareMathOperator{\onept}{1pt}

\newcommand{\rmS}{\mathrm S}
\newcommand{\lra}{\longrightarrow}

\newcommand{\supth}[1]{\ensuremath{#1^{\mathrm{th}}}}

\EpigaVolumeYear{7}{2023} \EpigaArticleNr{22} \ReceivedOn{November 21, 2022}
\InFinalFormOn{July 5, 2023}
\AcceptedOn{August 16, 2023}

\title{On the minimal model program for projective varieties with\\ pseudo-effective tangent sheaf}
\titlemark{MMP for  varieties with pseudo-effective tangent sheaf}

\author{Shin-ichi Matsumura}
\address{Mathematical Institute, Tohoku University, 
6-3, Aramaki Aza-Aoba, Aoba-ku, Sendai 980-8578, Japan.}
\email{mshinichi-math@tohoku.ac.jp}
\email{mshinichi0@gmail.com}

\authormark{S.~Matsumura}

\AbstractInEnglish{In this paper, we develop a theory of pseudo-effective sheaves on normal projective varieties.  As an application, by running the minimal model program, we show that projective klt varieties with pseudo-effective tangent sheaf can be decomposed into Fano varieties and Q-abelian varieties. }



\acknowledgement{This paper has been written during author's stay of University of Bayreuth.  The author would like to thank the members of University of Bayreuth for their hospitality.  He was partially supported by Grant-in-Aid for Scientific Research (B) $\sharp$21H00976 and Fostering Joint International Research (A) $\sharp$19KK0342 from JSPS.}

\begin{document}

%%%%%%%%%%%%%%%%%%%%%%%%%%%%%%%
% Title page
%%%%%%%%%%%%%%%%%%%%%%%%%%%%%%%

\maketitle

\begin{prelims}

\DisplayAbstractInEnglish

\bigskip

\DisplayKeyWords

\medskip

\DisplayMSCclass

\end{prelims}

%%%%%%%%%%%%%%%%%%%%%
% Table of Contents
%%%%%%%%%%%%%%%%%%%%%

\newpage

\setcounter{tocdepth}{1}

\tableofcontents

%%%%%%%%%%%%%%%%%%%%%
% Content begins here
%%%%%%%%%%%%%%%%%%%%%

\section{Introduction}\label{Sec-1}

\subsection{Motivation}

This paper aims to reveal the outcomes of the minimal model program (MMP) for projective klt varieties with pseudo-effective tangent sheaf.  The motivation of this paper lies in understanding the structure of projective varieties with certain non-negative curvature from the MMP viewpoint.

A smooth projective variety $X$ with pseudo-effective tangent bundle admits a smooth fibration $X \to Y$ onto an abelian variety $Y$ with rationally connected fibers (up to finite \'etale covers) by the main result of \cite{HIM}, which can be regarded as an extension of the main result of \cite{DPS94} formulated for nef tangent bundles.  The proofs of \cite{DPS94} and \cite{HIM} do not need the results of the MMP, but we can give another proof for the main result of \cite{DPS94} by using the MMP.  Indeed, \cite[Proposition 2.1]{CP91} and \cite[Section 5]{DPS94} assert that a smooth projective variety $X:=X_{0}$ with nef tangent bundle admits neither divisorial contractions nor flips.  Furthermore, a Mori fiber space $X=X_{0} \to X_{1}$ is a smooth fibration onto a smooth projective variety $X_{1}$ with nef tangent bundle.  Repeating this procedure for $X_{k}$, we obtain a sequence $X=X_{0} \to X_{1} \to \cdots \to X_{N}$ of Mori fiber spaces such that $X_{N}$ is one point or an \'etale quotient of an abelian variety.  The composite map $X=X_{0} \to X_{N}$ is also a Fano fibration by \cite[Theorem 5.3]{KW}, which re-proves the main result of \cite{DPS94} in the case where $X$ is projective.  Meanwhile, the MMP for projective varieties with pseudo-effective tangent bundle has not yet been studied.  More generally, although some structure theorems of varieties with certain non-negative curvature have recently been studied (for example, see \cite{CCM21, CH19, Mat20, Mat22, Wan22}), their relation with the MMP is still open for investigation.  As a first step, we reveal the MMP of projective varieties with pseudo-effective tangent bundle, which is the main motivation of this paper.

This paper has two specific purposes: The first purpose is to investigate what happens compared to the case of nef tangent bundles when we run the MMP for projective varieties with pseudo-effective tangent bundle.  This seems to be the first step toward understanding certain non-negative curvatures in the MMP.  The second purpose is to develop a basic theory of pseudo-effective torsion-free sheaves on normal projective varieties.  In our situation, the varieties appearing in the MMP can have singularities; therefore, the basic theory of pseudo-effective sheaves is actually needed.

\subsection{Main result}

The tangent sheaf $T_{X}$ of a normal projective variety $X$ is defined by the reflexive extension of the tangent bundle on the non-singular locus of $X$ (see Section~\ref{subsec-3-1} for the precise definition), and the pseudo-effectivity of $T_{X}$ is defined in Definition~\ref{def-psef} (see Proposition~\ref{prop-main} for characterizations of the pseudo-effectivity).  The following main result reveals the outcomes of the MMP for projective varieties with pseudo-effective tangent sheaf.

\begin{theo}\label{theo-main}
Let $X$ be a projective klt variety with pseudo-effective tangent sheaf.  Then, there exist finitely many projective varieties $\{X_{k}\}_{k=0}^{N}$ and $\{X'_{k}\}_{k=0}^{N}$ with
$$
X:=X_{0} \xdashrightarrow{\pi_{0}} X_{0}'  \xrightarrow{f_{0}} X_{1} \xdashrightarrow{\pi_{1}} X_{1}'\xrightarrow{f_{1}} \cdots 
%\xrightarrow{f_{k-1}} X_{k} \xdashrightarrow{\pi_{k}} X_{k}'  \xrightarrow{f_{k}} X_{k+1 } 
%\xrightarrow{\pi_{k+1}}
\cdots
\xrightarrow{f_{N-2}} X_{N-1} \xdashrightarrow{\pi_{N-1}} X_{N-1}'  \xrightarrow{f_{N-1}} X_{N } 
$$
satisfying the following conditions$:$ 
\begin{enumerate}
\item $X_{k}$ and   $X'_{k}$ are projective klt varieties with pseudo-effective tangent sheaf;
\item $\pi_{k}\colon X_{k} \dashrightarrow X'_{k}$ is a birational map 
obtained from the composite of divisorial contractions and flips; 
\item $f_{k}\colon X_{k} \to  X_{k+1}$ is a Mori fiber space; and 
\item $X_{N }$ is one point or a Q-abelian variety $($i.e., a quasi-\'etale quotient of an abelian variety$).$ 
\end{enumerate}
\end{theo}

Theorem~\ref{theo-main} is a structure theorem for a projective variety $X$ with pseudo-effective tangent sheaf, which says that the basic building blocks of $X$ are Fano varieties and Q-abelian varieties.  The theorem works not only for smooth varieties but also for klt varieties, which is an advantage compared to \cite{HIM}.  Note that $X$ can admit a divisorial contraction or a flip, although divisorial contractions or flips never appear in the case of nef tangent bundles.  Indeed, the blow-up $X:={\Bl}_{\onept}(Y) \to Y$ of a Hirzebruch surface $Y$ at a general point is a divisorial contraction, and the tangent bundle $T_{X}$ is pseudo-effective (see \cite[Section 4]{HIM}); also, smooth projective toric varieties, which always have pseudo-effective tangent bundle, can admit a flip (see \cite{Fuj03, FS04}).

The strategy of the proof of Theorem~\ref{theo-main} is as follows: We first run the MMP for $X$ using \cite[Corollary 1.3.3]{BCHM10} and then obtain a birational map $X \dashrightarrow X'$ and a Mori fiber space $X' \to Y$.  A key observation is that the pseudo-effectivity of the tangent sheaves is preserved by Propositions~\ref{prop-bir} and~\ref{prop-fib} (\textit{i.e.}, $T_{X'}$ and $T_{Y}$ are still pseudo-effective).  This follows from characterizations of the pseudo-effectivity (see Proposition~\ref{prop-main}).  This observation enables us to repeat this procedure for $Y$, leading us to obtain $\{X_{k}\}_{k=0}^{N}$ and $\{X'_{k}\}_{k=0}^{N}$ in Theorem~\ref{theo-main} so that $T_{X_{N}}$ is pseudo-effective and $K_{X_{N}}$ is nef.  We finally conclude that $X_{N }$ is actually (one point or) a $Q$-abelian variety by \cite[Theorem 1.2]{Gac}.

The remainder of this paper is organized as follows: In Section~\ref{Sec-2}, we develop a basic theory of pseudo-effective torsion-free sheaves on normal projective varieties, which is harder than we expected. In Section~\ref{Sec-3}, we study the MMP for projective varieties with pseudo-effective tangent sheaves to prove Theorem~\ref{theo-main}.

\subsection*{Notation}\label{subsec-not}
Throughout this paper, we interchangeably use the terms ``Cartier divisors,'' ``invertible sheaves,'' and ``line bundles.''  We also use the additive notation for tensor products (\textit{e.g.}, $L+M:=L\otimes M$ for line bundles $L$ and $M$).  Furthermore, we interchangeably use the terms ``locally free sheaves'' and ``vector bundles,'' and often simply abbreviate {\textit{possibly singular Hermitian}} metrics to ``metrics.''  All sheaves in this paper are coherent; thus, we omit the term ``coherent.''  Fibrations refer to proper surjective holomorphic maps with connected fibers.  We use the basic properties of the non-nef loci and the non-ample loci in \cite{BKK, Bou04, ELMNP06, ELMNP09}.

\subsection*{Acknowledgments}
The author would like to thank Prof.\ Kiwamu Watanabe for his question at the symposium on Algebraic Geometry at Waseda University, which gave an impetus to start studying the issue of this paper.  He also would like to thank Prof.\ Sho Ejiri for discussing \cite[Lemma 2.2]{EIM} and Prof.\ C\'ecile Gachet for discussing Proposition~\ref{thm-abel}.  
He is grateful to an anonymous referee for suggesting that he explain Example~\ref{ex}\eqref{ex-2}.

\section{Pseudo-effective sheaves on normal projective varieties}\label{Sec-2}

In this section, we develop a basic theory for the pseudo-effective torsion-free sheaves on normal projective varieties; specifically, we provide the definition of pseudo-effective sheaves and their fundamental properties.

\subsection{Singular Hermitian metrics on  torsion-free  sheaves}\label{subsec-2-1}
In this subsection, following \cite{MW}, we review singular Hermitian metrics on torsion-free sheaves, taking them on vector bundles as known (see \cite{Rau15, HPS18, PT18}).

Let $\cal{E}$ be a torsion-free (coherent) sheaf on a normal variety $X$.  Set $X_{0}:=X_{\reg} \cap X_\mathcal{E}$, where $X_{\reg}$ is the non-singular locus of $X$ and $X_\mathcal{E}$ is the maximal subset where $\mathcal{E}$ is locally free.  Note that $X_{0} \subset X$ is a Zariski open set with $\codim (X \setminus X_{0})\geq 2$ since $X$ is normal and $\cal{E}$ is torsion-free.  Let $h$ be \textit{a singular Hermitian metric} on $\cal{E}$, by which we mean a possibly singular Hermitian metric $h$ on the vector bundle $\mathcal{E}|_{X_{0}}$, where $\mathcal{E}|_{X_{0}}$ is the restriction of $\mathcal{E}$ to $X_{0}$.  Note that $h$ is a metric on the vector bundle $\mathcal{E}|_{X_{0}}$, but $h$ is not defined on $X \setminus X_0$.  Let $\theta$ be a smooth $(1,1)$-form on $X$ with local potential; that is, it can be written as $\theta=\ddbar f$ on a neighborhood of every point in $X$.  We then write
$$
\sqrt{-1}\Theta_{h} \geq \theta \otimes \id \text{ on } X  
$$
if for any local section $e \in H^0(U, \mathcal{E}^*)$ on an open set $U \subset X$, the function $\log |e|_{h^{*}}-f$ is plurisubharmonic (psh) on $U \cap X_0$, where $f$ is a local potential of $\theta$ and $h^{*}$ is the induced metric on the dual sheaf $\mathcal{E}^*:= \mathscr{H}\!\!\!\text{\calligra om}(\mathcal{E}, \mathcal{O}_{X})$.  The psh function $\log |e|_{h^{*}}-f$ is defined \textit{a priori} only on $U \cap X_0$, but it is automatically extended to a psh function on $U$ since $\codim (X \setminus X_{0})\geq 2$.  The condition $\sqrt{-1}\Theta_{h} \geq 0 \otimes \id $, simply written as $\sqrt{-1}\Theta_{h} \geq 0$ here, corresponds to the Griffiths semi-positivity of $(\cal E, h)$ when $\mathcal{E}$ is a vector bundle and $h$ is a smooth Hermitian metric.  We often write the condition as $\sqrt{-1}\Theta_{h} > 0$ if $X$ is compact and $\sqrt{-1}\Theta_{h} \geq \omega_{X}\otimes \id$ holds for some K\"ahler form $\omega_{X}$ on $X$ with local potential.

The following definition extends the notation of the pseudo-effectivity 
on vector bundles to  torsion-free  sheaves. 

\begin{defi}\label{def-psef}
Let $X$ be a compact K\"ahler space and $\omega_{X}$ be a K\"ahler form on $X$ with local potential.  A torsion-free sheaf $\mathcal{E}$ on $X$ is said to be {\textit{pseudo-effective}} if for every $m \in \mathbb{Z}_{+}$, there exists a singular Hermitian metric $h_{m}$ on the $\supth{m}$ symmetric power $\rmS^{m} \cal{E}|_{X_{0}}$ such that $\sqrt{-1}\Theta_{h_{m}}\geq -\omega_{X} \otimes \id$.
\end{defi}

\begin{rem}
Let $\mathcal{E}$ be a vector bundle on a smooth projective variety $X$, and consider the hyperplane bundle $\mathcal{O}_{\mathbb{P}(E) }(1)$ of the projective space bundle $\mathbb{P}(E) \to X$.  Even in this case, our definition of the pseudo-effectivity is stronger than the condition that $\mathcal{O}_{\mathbb{P}(E) }(1)$ is a pseudo-effective line bundle, which is often adopted as the definition of the pseudo-effectivity of $\mathcal{E}$.  Our definition requires that the image of the non-nef locus of $\mathcal{O}_{\mathbb{P}(E) }(1)$ is properly contained in $X$.
\end{rem}

Note that $h_{m}$ is a metric defined \textit{a priori} on $\rmS^{m} \mathcal{E}|_{X_{0}}$, but it can be extended to a metric on $\rmS^{m} \mathcal{E}|_{X_{\reg} \cap X_{\rmS^{m} \mathcal{E}}}$ since $\omega_{X}$ is defined on $X$ (not only on $X_{0}$).  The above-mentioned definition does not change even if we replace $\rmS^{m} \cal{E}|_{X_{0}}$ with the reflexive hull $\rmS^{[m]} \cal{E}:=(\rmS^{m} \cal{E})^{**}$.  Pseudo-effectivity can be defined in several other  ways.  These definitions are compared in Section~\ref{subsec-2-3}.

\subsection{Characterizations of pseudo-effective sheaves}\label{subsec-2-2}

In this subsection, we provide some characterizations of the pseudo-effectivity of torsion-free sheaves.  We first begin with fixing the notation.

\begin{setup}\label{setup}
Let $\cal{E}$ be a torsion-free sheaf on a normal projective variety $X$.  Let $\pi_{\mathcal{E}}\colon \bb{P}(\cal{E}) \to X$ be the main component of the projectivization ${\bf{Proj}}(\oplus_{m=0}^{\infty}\rmS^{m} \cal{E})$ of the graded sheaf $\oplus_{m=0}^{\infty}\rmS^{m} \cal{E}$ with the hyperplane bundle $\cal{O}_{\bb{P}(\cal{E})}(1)$, and let $\pi \colon P \to \bb{P}(\cal{E})$ be a resolution of singularities of $\bb{P}(\cal{E})$ via the normalization.  We have the following commutative diagram:
\[
\xymatrix{
\bb{P}(\cal{E}) \ar[d]_{\pi_{\mathcal{E}}}&  & P\ar[ll]_{\hspace{0cm} \pi}  \ar@/^16pt/[dll]^{p} \\
X\rlap{.} &  & 
}
\]
Set $X_{0}:=X_{\reg} \cap X_{\cal{E}}$ and $P_{0}:=p^{-1}(X_{0})$, where $X_{\cal{E}}$ is the maximal subset where $\cal{E}$ is locally free.  Assume that $\pi \colon P \to \bb{P}(\cal{E})$ is an isomorphism on $P_{0}=p^{-1}(X_{0})$ and that both the $\pi$-exceptional locus and $P \setminus P_{0}$ are divisorial.

The notation below is frequently used in this section:  
\begin{itemize}
\item $L:=\pi^{*} \cal{O}_{\bb{P}(\cal{E})}(1)$; 
\item $A$: an ample line bundle on $X$; 
\item $\omega_{P}$: a K\"ahler form on $P$; 
\item $\omega_{X}$: a K\"ahler form on $X$ with local potential; 
\item $\Lambda$: an effective $p$-exceptional  divisor such that 
$p_{*}(m(L +  \Lambda) )$ is reflexive for any $m \in \mathbb{Z}_{+}$. 
\end{itemize}
The existence of the divisor $\Lambda$ is guaranteed by \cite[Lemma III.5.10]{Nak04}.  As stated in Section~\ref{Sec-1}, the notation $p_{*}(M)$ refers to the direct image sheaf of the invertible sheaf $\mathcal{O}_{P}(M)$ associated to a divisor $M$.
\end{setup}

The following proposition characterizes the pseudo-effectivity of torsion-free sheaves.

\begin{prop}\label{prop-main}
We consider Setting~\ref{setup} and use the notation in Setting~\ref{setup} without explicit mention.  Then, the following conditions are equivalent:
\begin{enumerate}
\item\label{pm-1} There exists an ample line bundle $A$ on $X$ such that the reflexive hull $\rmS^{[m]} \cal{E} \otimes A$ is globally generated at a general point in $X$ for every $m \in \mathbb{Z}_{+}$.
\item\label{pm-2} There exists a K\"ahler form $\omega_{X}$ on $X$ with local potential satisfying the following: For every $m \in \mathbb{Z}_{+}$, there exists a singular Hermitian metric $h_{m}$ on $\rmS^{[m]} \cal{E}$ such that $\sqrt{-1}\Theta_{h_{m}}\geq -\omega_{X} \otimes \id$ on $X$ $($i.e., the sheaf $\mathcal{E}$ is pseudo-effective in the sense of Definition~\ref{def-psef}\,$)$.

\item\label{pm-3} The non-nef locus of $L |_{P_{0}}$ is not dominant over $X_{0}$ in the following sense:
For every $\e$, there exists a singular Hermitian metric $g_{\e}$ on $L|_{P_{0}}$ with the following:
\begin{itemize}
\item $\sqrt{-1}\Theta_{g_{\e}} \geq -\e \omega_{P}$ holds on $P_{0}$$;$
\item $\{x \in P_{0}\,|\, \nu(g_{\e}, x) >0 \}$ is not dominant over $X_{0}$; here $\nu(g_{\e}, x)$ denotes the Lelong number of the weight of $g_{\e}$.
\end{itemize}

\item\label{pm-4} Let $\Lambda$ be an effective $p$-exceptional divisor such that $p_{*}(m(L + \Lambda) )$ is reflexive for any $m \in \mathbb{Z}_{+}$.  The non-nef locus of $L + \Lambda $ is not dominant over $X$.

\item\label{pm-5} Let $\Lambda$ be an effective $p$-exceptional divisor such that $p_{*}(m(L + \Lambda) )$ is reflexive for any $m \in \mathbb{Z}_{+}$.  There exists an ample line bundle $A$ on $X$ such that the non-ample locus of $m(L+ \Lambda) +p^{*} A$ is not dominant over $X$ for every $m \in \mathbb{Z}_{+}$.

\item\label{pm-6} For an ample line bundle $A$ on $X$ and an integer $a \in \mathbb{Z}_{+}$, there exists an integer $b \in \mathbb{Z}_{+}$ such that the reflexive hull $\rmS^{[ab]} \cal{E} \otimes (bA)$ is globally generated at a general point in $X$.
\end{enumerate}
\end{prop}

\begin{proof}
\eqref{pm-1} $\Rightarrow$ \eqref{pm-2}.
By assumption, the sections of $\rmS^{[m]} \cal{E} \otimes A$ determine a singular Hermitian metric $H_{m}$ on $\rmS^{[m]} \cal{E} \otimes A$ with $\sqrt{-1}\Theta_{H_{m}}\geq 0 \otimes \id$ on $X$.  Since $A$ is ample, we can take a smooth Hermitian metric $g$ on $A$ such that $\omega_{X}:=\sqrt{-1}\Theta_{g}$ is a K\"ahler form with local potential.  We can then easily check that the metric $h_{m}:=H_{m} \otimes g^{-1}$ on $\rmS^{[m]} \cal{E}$ satisfies that $\sqrt{-1}\Theta_{h_{m}}\geq -\omega_{X} \otimes \id$ on $X$.

\eqref{pm-2} $\Rightarrow$ \eqref{pm-3}. 
Take a smooth Hermitian metric $g$ on $A$ such that $\sqrt{-1}\Theta_{g}$ is a K\"ahler form with local potential.  By replacing $(A, g)$ with $(kA, g^{k})$ for $k \gg 1$, we may assume that the metric $h_{m} g$ on $\rmS^{[m]} \mathcal{E} \otimes A$ satisfies that $\sqrt{-1}\Theta_{h_{m} g} \geq 0 \otimes \id $ on $X$ by assumption.

The fibration $p\colon P \to X$ over $X_{0}$ coincides with the projective space bundle $\mathbb{P}(\mathcal{E}) \to X$ of the locally free sheaf $\mathcal{E}|_{X_{0}}$.  In particular, the line bundle $L $ corresponds to $\mathcal{O}_{\mathbb{P}(\mathcal{E})}(1)$ over $X_{0}$; thus $L$ is relatively $p$-ample over $X_{0}$ and satisfies that $p_{*}(mL )=\rmS^{m} \mathcal{E}=\rmS^{[m]} \mathcal{E}$ on $X_{0}$.  This implies that the natural morphism
$$
p^{*} \left(\rmS^{[m]} \mathcal{E} \otimes A\right)=p^{*}p_{*} \left(mL  + p^{*}A\right) \lra mL  + p^{*}A
$$
is surjective over $X_{0}$ for any $m \gg 1$.  The metric $p^{*}(h_{m} g)$ defined on $p^{*} (\rmS^{[m]} \mathcal{E} \otimes A)|_{P_{0}}$ satisfies that
$$
\sqrt{-1}\Theta_{p^{*}(h_{m} g)} \geq 0 \otimes \id  \text{ on } P_{0}.
$$
Note that $P \setminus P_{0}$ may be divisorial; thus $p^{*}(h_{m} g)$ does not necessarily determine a metric on $X$.  Let us consider the singular Hermitian metric $G_{m}$ on $(mL + p^{*}A)|_{P_{0}}$ induced by $p^{*}(h_{m} g)$ and the above surjective morphism.  By construction, we see that $\sqrt{-1}\Theta_{G_{m}} \geq 0$ holds and the upper level set of Lelong numbers is not dominant over $X_{0}$.  The metric $g_{m}:=(G_{m} p^{*}g)^{1/m} $ on $L|_{P_{0}}$ satisfies that $\sqrt{-1}\Theta_{g_{m}} \geq -(1/m) p^{*}\omega_{X}$.  We can then easily see that the metrics $\{ g_{m} \}_{m=1}^{\infty}$ on $L|_{P_{0}}$ for $m \gg 1$ provide the desired metrics $\{ g_{\e} \}_{\e>0}$.

\eqref{pm-3} $\Rightarrow$ \eqref{pm-4}. 
Fix an effective $p$-exceptional divisor $\Lambda$ such that $p_{*}(m(L + \Lambda) )$ is reflexive.  Almost all points $y \in Y_{0}$ satisfy that
$$
\mathcal{I} \left(g^{m}_{\e}|_{X_{y}}\right) =\mathcal{I} \left(g^{m}_{\e}\right)|_{X_{y}}  = \mathcal{O}_{P_{y}} \text{ holds 
for any } m \in \mathbb{Z}_{+}
$$
by Fubini's theorem and the restriction formula (see \cite[the argument of Claim 2.1]{Mat18} for the precise argument).  Here $\mathcal{I} (g_{\e})$ is the multiplier ideal sheaf, and $P_{y}$ is the fiber of $p\colon P \to X$ at $y \in X$.  Note that the last equality follows from the assumption on Lelong numbers.  We fix such a point $y$ with the above property.  The fiber $P_{y} $ does not intersect with the $p$-exceptional divisor $\Lambda$; in particular, we obtain $(m(L+ \Lambda ) + p^* A)|_{P_{y}} = mL |_{P_{y}}$.

For a sufficiently ample line bundle $A$, we will prove that the restriction map
\begin{align}\label{sur}
H^{0}(P, m(L+ \Lambda )+ p^* A) \lra 
H^{0}\left(P_{y},  (m(L+ \Lambda ) + p^* A)|_{P_{y}}\right)
= H^{0}\left(P_{y},  mL |_{P_{y}}\right)
\end{align}
is surjective for $m \gg 1$.  We now check that condition~\eqref{pm-4} follows from this surjectivity.  To this end, we consider the singular Hermitian metric $G_{m}$ on $m(L+ \Lambda) + p^* A$ induced by extensions of a basis of $H^{0}(P_{y}, mL|_{P_{y}})$.  The fibration $p\colon P \to X$ coincides with the projective space bundle $\mathbb{P}(\mathcal{E}) \to X$ over $X_{0}$; hence $mL|_{P_{y}}$ is very ample. Thus the metric $G_{m}$ is smooth on a neighborhood of $P_{y}$.  This indicates that for a smooth metric $g$ on $A$, the metrics $g_{m}:=(G_{m} p^{*}g)^{1/m}$ provide the desired metrics on $L+\Lambda$; therefore, the non-nef locus of $L+\Lambda$ is not dominant over $X$ (see \cite[Definition 3.3]{Bou04}).

To extend sections on the fiber $P_{y}$, we first extend them to the Zariski open set $P_{0}=p^{-1}(X_{0})$ by using a version of the Ohsawa--Takegoshi $L^{2}$-extension theorem (see Lemma~\ref{lem-key2}).  Lemma~\ref{lem-key2} will be proved later.  For a sufficiently ample line bundle $A$ on $X$, the line bundle $\cal{O}_{P(\mathcal{E})}(1) + \pi_{\cal E}^{*}A$ is ample on $P(\mathcal{E})$ since $\cal{O}_{P(\mathcal{E})}(1)$ is relatively $\pi_{\mathcal{E}}$-ample.  This implies that the non-ample locus of the line bundle
$$
L+p^* A=\pi^{*} \left(\cal{O}_{P(\mathcal{E})}(1) + \pi_{\cal E}^{*}A\right)
$$ 
is contained in the $\pi$-exceptional locus.  Hence, we find an ample line bundle $A_{P} $ on $P$ and an effective $\pi$-exceptional divisor $E$ such that $k_{0}(L+p^* A) = A_{P} + E$ holds and $A_{P} - K_{P}$ is ample.  We will show that the restriction map
\begin{align}\label{sur2}
H^{0}\left(P_{0}, m(L+\Lambda) + k_{0} p^* A\right) \lra 
H^{0}\left(P_{y}, (m(L+\Lambda) + k_{0} p^* A)|_{P_{y}}\right) 
= H^{0}\left(P_{y}, mL|_{P_{y}}\right) 
\end{align}
is surjective for any $m \gg 1$.  We define the line bundle $M$ by
\begin{align*}
M:=(m-k_{0}) L+ \left(A_{P}- K_{P}\right) + E+m\Lambda \text{ so that }m(L+\Lambda) + k_{0} p^* A = K_{P} + M 
\end{align*}
and equip $M$ with the metric $G:=g_{\e}^{m-k_{0}} g g_{E+m\Lambda}$, where $g_{\e}$ is the metric in condition~\eqref{pm-3}, $g$ is a smooth Hermitian metric on $A_{P}- K_{P}$ with $\sqrt{-1}\Theta_{g} > 0$, and $g_{E+m\Lambda}$ is the singular Hermitian metric induced by the natural section of the effective divisor $E+m\Lambda$.  By construction, we see that $\sqrt{-1}\Theta_{G} > 0$ for any $1 \gg \e > 0$.  Let $\psi$ be a quasi-psh function on $P$ with neat analytic singularities such that the subvariety $V$ defined by $\mathcal{O}_{P}/\I{\psi}$ is $P_{y}$ (see \cite[Definition (2.2)]{Dem16} for neat analytic singularities).  We ensure that the curvature $\sqrt{-1}\Theta_{G} $ satisfies assumption~\eqref{lk2-2}  in Lemma~\ref{lem-key2} by taking $A_{P} $ to be sufficiently ample.  Furthermore, we obtain $\I{G|_{X_{y}}}=\mathcal{O}_{P_{y}}$ by the choice of $y$ and $P_{y} \cap \Supp (E+\Lambda)=\emptyset$.  Hence, by Lemma~\ref{lem-key2}, the restriction map \eqref{sur2} is surjective.

We finally extend sections on $P_{0}$ to $P$.  Since $\codim (X\setminus X_{0}) \geq 2$, we obtain
\begin{align*}
H^{0}(P_{0}, m(L+\Lambda) +k_{0} p^{*}A)\ = \ &H^{0}(X_{0}, p_{*} (m(L+\Lambda) \otimes  k_{0} A)\\
\cong \ &H^{0}(X, p_{*} (m(L+\Lambda) \otimes  k_{0} A)\\
= \ &H^{0}(P, m(L+\Lambda) +k_{0} p^{*}A). 
\end{align*}
Here we use the reflexivity of $p_{*}(m(L+\Lambda))$ to obtain the above isomorphism.  Therefore, the restriction map \eqref{sur} is surjective, finishing the proof.

\eqref{pm-4} $\Rightarrow$ \eqref{pm-5}.   
By the same way as in the proof of \eqref{pm-3} $\Rightarrow$ \eqref{pm-4}, we find an ample line bundle $A_{P} $ on $P$ and an effective $\pi$-exceptional divisor $E$ such that $k_{0}(L+p^* A) = A_{P} + E$ holds.  The non-ample locus $(m-k_{0}) (L + \Lambda) + A_{P}$ is not dominant over $X$ by assumption.  Hence, condition~\eqref{pm-5} follows from
$$
m (L+\Lambda) + k_{0} p^{*} A= (m-k_{0})  (L + \Lambda) + A_{P}
+ k_{0} \Lambda   + E. 
$$

\eqref{pm-5} $\Rightarrow$ \eqref{pm-1}.  
Let $y$ be a general point in $X$.  The fiber $P_{y}$ does not intersect with the non-ample locus of $m (L+\Lambda) + p^{*}A$ since the non-ample locus is a Zariski closed set that is not dominant over $X$ by assumption.  Therefore, we can take a singular Hermitian metric $g$ such that $\sqrt{-1}\Theta_{g}>0$ holds and $g$ is smooth on a neighborhood of the fiber $P_{y}$.  By considering the multiple of $m (L + \Lambda) + p^{*}A$, we may assume that $\sqrt{-1}\Theta_{g}$ is sufficiently positive such that the restriction map
\begin{align*}
H^{0}\left(P, m (L+\Lambda) + p^{*}A\right) \lra H^{0}\left(P_{y}, m (L+\Lambda)|_{P_{y}}\right) 
\end{align*}
is surjective, by the standard extension theorem 
(for example, see \cite[Theorem 1.1]{CDM17} and the proof of \cite[Proposition 4.1]{CCM21}).
This implies that 
$$
p_{*}(m(L +\Lambda) + p^{*}A)=\rmS^{[m]} \mathcal{E} \otimes A
$$
is globally generated at $y$, finishing the proof. 

\eqref{pm-1} $\Rightarrow$ \eqref{pm-6}. 
This implication is obvious. 

\eqref{pm-6} $\Rightarrow$ \eqref{pm-3}.
The proof is almost the same as in that for  \eqref{pm-2} $\Rightarrow$ \eqref{pm-3}.  
The natural morphism 
$$
p^{*} \left(\rmS^{[ab]} \mathcal{E} \otimes (bA)\right) \lra abL  + p^{*}(bA)  
$$
is surjective over $ X_{0}$. 
By assumption, for an integer $a \in \mathbb{Z}_{+}$, we can take an integer $b \in \mathbb{Z}_{+}$ such that $ \rmS^{[ab]} \mathcal{E} \otimes (bA)$ is globally generated at a general point.  In the same way as in the proof for \eqref{pm-2} $\Rightarrow$ \eqref{pm-3}, we see that the induced singular Hermitian metric $G_{a}$ on $abL + p^{*}(bA)|_{P_{0}}$ is smooth along the fiber at a general point and satisfies that $\sqrt{-1}\Theta_{G_{a}} \geq 0$.  Take a smooth Hermitian metric $g$ on $A$ such that $\sqrt{-1}\Theta_{g}$ is a K\"ahler form with local potential.  Then, the metrics $\{(G_{a})^{1/ab} (p^{*}g)^{-1/a}\}_{a\in \mathbb{Z}_{+}}$ provide the desired metrics.
\end{proof}

The following lemma, known to experts, easily follows from the Ohsawa--Takegoshi $L^2$-extension theorem (see \cite{OT87, Man93}).  We give an outline of the proof for the convenience of the reader.

\begin{lemm}\label{lem-key2}
Let $M$ be a line bundle on a smooth projective variety $P$, and let $Z \subset P$ be a Zariski closed subset of\, $P$.  Set $P_{0}:=P\setminus Z$.  Let $h$ be a singular Hermitian metric on $M|_{P_{0}}$ and $\psi$ be a quasi-psh function on $P$ with neat analytic singularities.  We assume the following conditions$:$

\begin{enumerate}
\item\label{lk2-1} The subvariety $V$ defined by\, $\mathcal{O}_{P}/\I{\psi}$ is smooth and satisfies that $V \subset P_{0}$.

\item\label{lk2-2} The inequality $\sqrt{-1}\Theta_{h}+ (1+\delta) \deldel \psi \geq 0$ holds on $P_{0}$ for any $1\gg \delta>0$.
\end{enumerate}

Then, for a section $f \in H^{0}(V, (K_{P}+M)|_{V}\otimes \I{h|_{V}})$, there exists a section $F \in H^{0}(P_{0}, (K_{P}+M)|_{P_{0}})$ such that $F|_{V}=f$.
\end{lemm}

\begin{proof}
In the case where $P_{0}$ is weakly pseudoconvex, this theorem directly follows from the Ohsawa--Takegoshi $L^2$-extension theorem.  For example, see \cite[(2.8) Theorem]{Dem16} (\textit{cf.}\ \cite{CDM17, ZZ20}) for a formulation similar to this theorem.

The Zariski open set $P_{0}$ is not necessarily weakly pseudoconvex, but we can reduce the proof to this case by the projectivity of $P$.  Indeed, by the projectivity, we can find a smooth hypersurface $H \subset P$ such that $P\setminus H$ is Stein and that $Z \subset H$ and $V \not \subset H$ hold.  Note that $P_{0}\setminus H=P\setminus H$ is weakly pseudoconvex.  Hence, the section $f|_{V \setminus H} \in H^{0}(V \setminus H, (K_{P}+M)|_{V}\otimes \I{h|_{V}})$ is extended to a section $F \in H^{0}(P_{0} \setminus H, (K_{P}+M)|_{P_{0} \setminus H})$ whose $L^{2}$-norm of $F$ with respect to $h$ on $P_{0} \setminus H$ converges.  Fixing a local frame of $K_{P}+M$, we regard $F$ as a holomorphic function locally defined on $P_{0}\setminus H$.  For every point $p \in H \setminus Z$, since the local weigh of $h$ is quasi-psh, the metric $h$ is bounded below on an neighborhood of $p$; thus, the $L^2$-norm of the holomorphic function $F$ converges.  This indicates that $F$ is extended through $H \setminus Z$ by the $L^{2}$-boundedness.  (Note that $F$ is not necessarily extended through $Z$ since $h$ may not be bounded below on a neighborhood of a point in~$Z$.)
\end{proof}

\subsection{Fundamental properties of pseudo-effective sheaves}\label{subsec-2-3}

In this subsection, we provide fundamental properties of pseudo-effective sheaves and compare Definition~\ref{def-psef} to other possible ways to define the pseudo-effectivity.

We first examine the behavior of the pseudo-effectivity for the pull-back.  Let $f\colon X \to Y$ be a fibration between normal projective varieties.  A vector bundle $E$ on $Y$ is nef (resp.\ pseudo-effective) if and only if $f^{*}E$ is nef (resp.\ pseudo-effective).  Let $\mathcal{E}$ be a pseudo-effective torsion-free sheaf on $Y$.  Then, the pull-back $f^{*}\mathcal{E}$ is not necessarily torsion-free.  Even if we consider the quotient $ (f^{*} \mathcal{E}/\tor)$ by the torsion subsheaf of $f^{*}\mathcal{E}$, it is not pseudo-effective in general (see Example~\ref{ex} below).  However, Proposition~\ref{prop-basic} below shows that the converse implication is true; that is, the sheaf $\mathcal{E}$ is pseudo-effective if $ (f^{*} \mathcal{E}/\tor)$ is pseudo-effective.  Proposition~\ref{prop-basic} is applied when we prove Theorem~\ref{theo-main} or compare Definition~\ref{def-psef} to other definitions of the pseudo-effectivity.

\begin{prop}\label{prop-basic}
Let $f\colon X \dashrightarrow Y$ be an almost holomorphic map between normal projective varieties, and let $\cal E$ and $\cal F$ be torsion-free sheaves on $X$ and $Y$, respectively.  Assume that there exists a Zariski open set $Y_{0} \subset Y$ with $\codim (Y \setminus Y_{0}) \geq 2$ such that
\begin{itemize}
\item $f\colon X \dashrightarrow Y$ is an $($everywhere defined\,$)$ fibration over $Y_{0}$ and 
\item $f^{*} \cal F=\cal E$ holds over $Y_{0}$. 
\end{itemize}
Then, the sheaf $\cal F$ is pseudo-effective if $\cal E$ is pseudo-effective. 
\end{prop}

\begin{proof}
We assume that $\cal F$ is locally free on $Y_{0}$ by replacing $Y_{0}$ with $Y_{0} \cap Y_{\cal F}$, where $Y_{\cal F}$ is the maximal locally free locus of $\cal F$.

Let $y$ be a general point in $Y_{0}$.  Let $A$ and $B$ be ample Cartier divisors on $X$ and $Y$, respectively.  By assumption, for an integer $a \in \mathbb{Z}_{+}$, there exists an integer $b \in \mathbb{Z}_{+}$ such that
$$
\Bs_{(a,b)}(\mathcal{E})
:=\left\{x \in X \,|\, \text{the stalk of }  \rmS^{[ab]} \cal{E} \otimes (bA) \text{ at } x 
\text{ is not globally generated}
\right\}
$$
is a proper Zariski closed set in $X$.  From this condition, we will show that for any $a \in \mathbb{Z}_{+}$, there exists an integer $b \in \mathbb{Z}_{+}$ such that the stalk of $\rmS^{[ab]} \cal{F} \otimes(bB)$ at $y\in Y$ is generated by a section in $H^{0}(Y_{0}, \rmS^{[ab]} \cal{F} \otimes(bB))$.  This finishes the proof by condition~\eqref{pm-6} in Proposition~\ref{prop-main} since such a section is automatically extended to $Y$ by the reflexivity and since $\codim (Y \setminus Y_{0}) \geq 2$.  To this end, following \cite[Lemma 2.2]{EIM}, we will reduce our situation to the case where $f\colon X \dashrightarrow Y$ is an everywhere defined and generically finite morphism such that $X_{y}:=f^{-1}(y)$ does not intersect with $\Bs_{(a,b)}(\mathcal{E})$.

We may assume that $f $ is an everywhere defined fibration by replacing $f\colon X \dashrightarrow Y$ with $f\colon X_{0}:=f^{-1}(Y_{0}) \to Y_{0}$.  Note that $\Bs_{(a,b)}(\mathcal{E})$ is still a proper Zariski closed set in $X$ since $\Bs_{(a,b)}(\mathcal{E}|_{X_{0}}) \subset \Bs_{(a,b)}(\mathcal{E}) \cap X_{0}$.  Both $X$ and $Y$ are non-compact, but this does not affect in the argument below.

We now check that we may assume that $f\colon X \to Y$ is a generically finite morphism.  Let $k$ be the fiber dimension of $f\colon X \to Y$.  Since $y$ is a general point, we see that $\dim( \Bs_{(a,b)}(\mathcal{E}) \cap X_{y}) < k$ and the fibration $f\colon X \to Y$ is flat over $y$.  For general hypersurfaces $\{H_{i} \}_{i=1}^{k}$ on $X$, we replace $X$ with the complete intersection $X':=X \cap H_{1} \cap \cdots \cap H_{k}$.  Then, since $k$ is the fiber dimension of $f\colon X \to Y$, the replaced fibration $f\colon X \to Y$ is a generically finite morphism.  Note that $f\colon X \to Y$ is flat over $y$; furthermore, the fiber $X_{y}$ does not intersect with $\Bs_{(a,b)}(\mathcal{E}) $ since $\Bs_{(a,b)}(\mathcal{E}|_{X'}) \subset \Bs_{(a,b)}(\mathcal{E}) \cap X' $ and $\dim( \Bs_{(a,b)}(\mathcal{E}) \cap X_{y}) < k$.

The generically finite morphism $f\colon X \to Y$ is finite at $y$; hence we may assume that $A$ and $B$ are effective divisors and $X_{y} \cap \mathrm{Supp}(g^*B-A) = \emptyset$ by replacing the ample Cartier divisors $A$ and $B$ if necessary.  By the definition of $\Bs_{(a,b)}(\mathcal{E}) $ and the relation $ X_{y} \cap \Bs_{(a,b)}(\mathcal{E}) = \emptyset$, the sheaf $\rmS^{[ab]} \cal{E} \otimes (bA)$ is globally generated at any points in $X_y$; hence so is $\rmS^{[ab]} \cal{E} \otimes (bf^{*}B)$ since $X_{y} \cap \mathrm{Supp}(g^*B-A) = \emptyset$.  Thus, we obtain a morphism that is surjective on $X_{y}$:
$$\alpha\colon  \bigoplus  \mathcal O_X \lra   \rmS^{[ab]} \cal{E} \otimes (bf^{*}B). 
$$ 
Since $f\colon X \to Y$ is affine over a neighborhood of $y$, the morphism induced by the push-forward
$$
\beta \colon  
\bigoplus f_*\mathcal O_X 
\xrightarrow{f_*(\alpha)}
f_*\mathcal O_X\left(\rmS^{[ab]} \cal{E} \otimes (bf^{*}B)\right)
\cong \left(f_*\mathcal O_X\right) \otimes \mathcal O_Y\left(\rmS^{[ab]} \cal{F} \otimes (bB)\right)
$$
is surjective at $y$.  Here, the isomorphism on the right-hand side follows from the projection formula and $\rmS^{[ab]} \cal{E}=f^{*}\rmS^{ab} \cal{F}$ by noting that we have already replaced the original variety $Y$ with $Y_{0}$.  Furthermore, since $f_*\mathcal O_X$ is locally free at $y$, the natural pairing
$$
\gamma \colon  (f_*\mathcal O_X )^* \otimes f_*\mathcal O_X \lra \mathcal O_Y
$$
is surjective at $y$.  
Take $n\in\mathbb Z_{+}$ such that 
$
\left( (f_*\mathcal O_X)^* \otimes f_*\mathcal O_X \right)\otimes(nB)
$
is globally generated. 
The above argument implies that the following morphism  is surjective at $y$: 
\begin{align*}
\bigoplus 
\left( (f_*\mathcal O_X)^{*} 
\otimes 
f_*\mathcal O_X \right)\otimes(nB)
\ \cong \ &
(f_*\mathcal O_X)^{*} 
\otimes 
\left( \bigoplus  f_*\mathcal O_X \right)\otimes(nB)
\\ \xrightarrow{\textup{induced by $\beta$}}  
& \ \left( (f_*\mathcal O_X)^{*} \otimes f_*\mathcal O_X \right)\otimes
\left( \rmS^{[ab]} \cal{F} \otimes ((b+n)B)) \right) 
\\ \xrightarrow{\textup{induced by $\gamma$}}  
& \ \mathcal \rmS^{[ab]} \cal{F} \otimes \left( (b+n)B \right). 
\end{align*}
Hence, the stalk of $\rmS^{[ab]} \cal{F} \otimes((b+n)B)$ at $y$ is generated by global sections, finishing the proof.
\end{proof}

In the remainder of this subsection, we observe other possible ways to define the pseudo-effectivity.  One approach of defining the pseudo-effectivity of a torsion-free sheaf $\mathcal{E}$ is to use a birational morphism $\alpha \colon \tilde{X} \to X$ such that the quotient $ (\alpha^{*} \mathcal{E}/\tor)$ by the torsion subsheaf of the pull-back $\alpha^{*} \mathcal{E}$ is locally free.  Another approach is to use $L=\pi^{*} \cal{O}_{\bb{P}(\cal{E})}(1)$ instead of $L+\Lambda$ in Setting~\ref{setup}.  The following proposition shows that these definitions are stronger than Definition~\ref{def-psef}.

\begin{prop}\label{prop-cor}
Let $\mathcal{E}$ be a torsion-free sheaf on a normal projective variety $X$.
\begin{enumerate}
  \item\label{pc-1} If the non-nef locus of\, $L$ is not dominant over $X$, then $\mathcal{E}$ is pseudo-effective.

\item\label{pc-2} Let $\alpha \colon \tilde{X} \to X$ be a birational morphism such that the quotient $ (\alpha^{*} \mathcal{E}/\tor)$ by the torsion subsheaf of the pull-back $\alpha^{*} \mathcal{E}$ is locally free.  If\, $(\alpha^{*} \mathcal{E}/\tor)$ is pseudo-effective, then $\mathcal{E}$ is pseudo-effective.
\end{enumerate}
\end{prop}

\begin{proof}
Conclusion~\eqref{pc-1} follows from $ \mathbb{B}_{-}(L+\Lambda) \subset \mathbb{B}_{-}(L) \cup \Lambda$ and condition~\eqref{pm-4} in Proposition~\ref{prop-main}.  Conclusion~\eqref{pc-2} is a direct consequence of Proposition~\ref{prop-basic}.
\end{proof}

The following examples show that the converse implications of Proposition~\ref{prop-cor} are not true in general.

\begin{ex}\label{ex}\leavevmode
\begin{enumerate}
\item\label{ex-1} Let $X$ be a smooth projective variety.  We consider the ideal sheaf $\mathcal{E}:=\mathcal{I}_{Z}$ defined by a smooth subvariety $Z \subset X$ of codimension at least $2$ and the blow-up $\alpha \colon \tilde{X} \to X$ along $Z$.  Then, the quotient $ (f^{*} \mathcal{E}/\tor)$ by the torsion subsheaf is the invertible sheaf $\mathcal{O}_{\tilde X}(-E)$ associated to an effective $\alpha$-exceptional divisor $E$.  The sheaf $\mathcal{E}:=\mathcal{I}_{Z}$ is obviously pseudo-effective since $\rmS^{[ab]} (\mathcal{I}_{Z})=\mathcal{O}_{X}$, but $\mathcal{O}_{\tilde X}(-E)$ is not pseudo-effective.  The blow-up $\alpha \colon \tilde{X} \to X$ along $Z$ coincides with $\mathbb{P}(\mathcal{E}) \to X$; hence $P$ in Setting~\ref{setup} can be chosen to be $P=\tilde{X}=\mathbb{P}(\mathcal{E})$.  Furthermore, we see that $\mathcal{O}_{\mathbb{P}(\mathcal{E}) }(1)=\mathcal{O}_{\mathbb{P}(\mathcal{E})}(-E)$ and $\Lambda=E$.  Then the line bundle $L+\Lambda$ is trivial (and thus pseudo-effective), but $\mathcal{O}_{\mathbb{P}(\mathcal{E}) }(1)=\mathcal{O}_{\mathbb{P}(\mathcal{E})}(-E)$ is not pseudo-effective.

\item\label{ex-2}  This example is due to \cite[Remark 2.7]{Gac}: Let $\cal E$ be the tangent sheaf $T_X$ of a singular Kummer surface $X$ in \cite[Remark 2.7]{Gac}.  Then, there exists a sheaf $\mathcal{F}$ on $X$ such that $\cal E = \cal F \oplus \cal F$ and $\cal F^{\otimes 2} = \mathcal{I}_{X_{\sing}}$; hence, the reflexive hull $\rmS^{[2a]}(\cal E)$ is a trivial vector bundle, which indicates that $\cal E$ is pseudo-effective.  Nevertheless, since $\cal F^{\otimes 2} = \mathcal{I}_{X_{\sing}}$ and by the same argument as in~\eqref{ex-1}, we see that neither $(\alpha^{*} \mathcal{E}/\tor)$ nor $L$ is pseudo-effective.
\end{enumerate}

\end{ex}

We finally consider the pseudo-effectivity of $\bb Q$-Cartier divisors on normal projective varieties.

\begin{prop}\label{prop-cor2}
Let $D$ be a Weil divisor on a normal projective variety $X$ and $\mathcal{E}$ be the sheaf associated to the Weil divisor $D$.  Assume that $D$ is $\bb Q$-Cartier.  Then, the sheaf $\mathcal{E}$ is pseudo-effective in the sense of Definition~\ref{def-psef} if and only if $D$ is pseudo-effective as a $\bb Q$-Cartier divisor.
\end{prop}

\begin{proof}
Recall that $D$ is said to be {\textit{pseudo-effective}} (as a $\bb Q$-Cartier divisor) if there exist an ample line bundle $A$ 
and an integer $m_{0} \in \mathbb{Z}_{+}$ with  $m_{0}D$ Cartier such that $km_{0}D+A$ has a non-zero section for any $k \in \mathbb{Z}_{+}$.

Fix an integer $m_{0} \in \mathbb{Z}_{+}$ with  $m_{0}D$ Cartier.  Then, we have $\rmS^{[km_{0}]} \mathcal{E} \cong \mathcal{O}_{X}(km_{0} D)$.  Hence, condition~\eqref{pm-1} in Proposition~\ref{prop-main} implies that $D$ is pseudo-effective as a $\bb Q$-Cartier divisor.

To prove the converse implication, we take an ample line bundle $A$ such that $km_{0}D+A$ has a non-zero section for any $k \in \mathbb{Z}_{+}$.  We may assume that $\rmS^{[r]} \mathcal{E}\otimes A $ is globally generated for any $0 \leq r < m_{0}$.  For a given integer $m \in \mathbb{Z_{+}}$, after taking $q$ and $r$ such that $m=q m_{0}+r$ and $0 \leq r < m_{0}$, we obtain
$$\rmS^{[m]} \mathcal{E} =\mathcal{O}_X(qm_{0} D )\otimes 
\rmS^{[r]} \mathcal{E}. 
$$ 
Therefore, the sheaf $\rmS^{[m]} \mathcal{E} \otimes 2A$ has a non-zero section; thus it is generically globally generated.
\end{proof}

\section{MMP for varieties with pseudo-effective tangent sheaf}\label{Sec-3}

\subsection{Fibrations and pseudo-effective tangent sheaves}\label{subsec-3-1}

In this subsection, we consider the behavior of the pseudo-effectivity of tangent sheaves under birational maps or fibrations.  The tangent sheaf $T_{X}$ of a normal variety $X$ is defined by the reflexive hull:
$$T_{X}:=\left( j_{*} T_{X_{\reg}} \right)^{**} :=
\left( j_{*} \mathcal{O}_{X_{\reg}}(T_{X_{\reg}}) \right)^{**}, 
$$
where $T_{X_{\reg}}$ is the tangent bundle on the non-singular locus $X_{\reg}$ and $j\colon X_{\reg} \to X$ is the natural inclusion.  Note that $ (\pi_{*}T_{\tilde{X}})^{**} = T_{X}$ holds for any resolution $\alpha\colon \tilde {X} \to X$ of singularities of $X$.

The following propositions essentially follow from Proposition~\ref{prop-basic}.

\begin{prop}\label{prop-bir}
Let $X \dashrightarrow Y$ be a birational map between normal projective varieties.  Then, if the tangent sheaf $T_{X} $ of\, $X$ is pseudo-effective, so is the tangent sheaf\, $T_{Y} $ of\, $Y$.
\end{prop}

\begin{prop}\label{prop-fib}
Let $f\colon X \to Y$ be a fibration between normal projective varieties.  If the tangent sheaf\, $T_{X} $ of\, $X$ is pseudo-effective, so is the tangent sheaf\, $T_{Y} $ of\, $Y$.
\end{prop}

\begin{proof}[Proofs of Propositions~\ref{prop-bir} and~\ref{prop-fib}]
Proposition~\ref{prop-basic} is formulated for almost holomorphic maps; thus Propositions~\ref{prop-bir} is a direct consequence of Proposition~\ref{prop-basic}.

For the proof of Proposition~\ref{prop-fib}, we take resolutions $\tilde X \to X$ and $\tilde Y \to Y$ of singularities of $X$ and $Y$ with the following commutative diagram:
$$
\xymatrix{
    \tilde{X} \ar[r]^{\alpha} \ar[d]_{\tilde f} & X \ar[d]^{f} \\
    \tilde{Y} \ar[r]_{\beta}       & Y.}
$$
Set $Y_{0}:= Y \setminus \beta(E)$, where $E$ is the $\beta$-exceptional locus.  Then, we obtain $ f^{*}T_{Y}={\alpha}_{*} \tilde f^{*} T_{\tilde Y}$ on $X_{0}:=f^{-1}(Y_{0})$ from 
$$
{\alpha}^{*}f^{*}T_{Y}=   \tilde{f}^{*} \beta^{*} T_{Y}  = \tilde{f}^{*}  T_{\tilde Y} \text{ on } X_{0}. 
$$
Meanwhile, the natural sheaf morphism $T_{\tilde {X}}\to \tilde f^{*} T_{\tilde Y}$ is generically surjective; hence so is the induced morphism
$$
T_{X}=\left( {\alpha}_{*} T_{\tilde {X}}   \right)^{**} \lra  \left( {\alpha}_{*} \tilde f^{*} T_{\tilde Y} \right)^{**}. 
$$
The quotient of pseudo-effective sheaves by generically surjective morphisms is also pseudo-effective; thus $( {\alpha}_{*} \tilde f^{*} T_{\tilde{Y}} )^{**}$ is a pseudo-effective sheaf, and it coincides with $f^{*}T_{Y}$ on $X_{0}=f^{-1}(Y_{0})$.  Hence the conclusion follows from Proposition~\ref{prop-basic}.
\end{proof}

\subsection{Outcomes of the MMP for varieties with pseudo-effective tangent sheaf}\label{subsec-3-2}

We finally prove Theorem~\ref{theo-main} after checking the following propositions.

\begin{prop}[\protect{\textit{cf.}\ {\cite[Theorem 1.2]{Gac}}}]\label{thm-abel}
Let $X$ be a projective klt variety.  If the tangent sheaf $T_{X}$ is pseudo-effective and the canonical divisor $K_{X}$ is numerically trivial, then $X$ is a Q-abelian variety.
\end{prop}

\begin{proof}
Condition~\eqref{pm-1} in Proposition~\ref{prop-main} shows that our definition of pseudo-effective sheaves is stronger than \cite[Definition 2.10]{Gac}.  Hence, by \cite[Theorem 1.2]{Gac}, there exists a finite quasi-\'etale cover $X' \to X$ such that $X'$ is the product $A \times Y$ of an abelian variety $A$ and  a projective variety~$Y$.  Since $X' \to X$ is quasi-\'etale, the tangent sheaf $T_{X'}$ is pseudo-effective, and so is $T_{Y}$.  This part is valid for the pseudo-effectivity in the sense of Definition~\ref{def-psef}, but not in the sense of \cite[Definition 2.10]{Gac}.  Furthermore, we can easily see that $Y$ is a projective klt variety with numerically trivial canonical divisor.  Therefore, by using the induction hypothesis on the dimension, we see that the variety $Y$ is Q-abelian, and so is $X$.
\end{proof}

\begin{prop}\label{det}
Let $\mathcal{E}$ be a pseudo-effective sheaf on a compact K\"ahler space $X$.  Then, the sheaf\, $\det \mathcal{E}:=(\Lambda^{r} \mathcal{E})^{**} $ is pseudo-effective.  Here $r$ is the rank of $\mathcal{E}$.  In particular, when the sheaf\, $\det \mathcal{E}$ is $\mathbb Q$-Cartier, it is pseudo-effective as a $\mathbb Q$-Cartier divisor.
\end{prop}

\begin{proof}
It is sufficient to construct singular Hermitian metrics $h_{m}$ on $\det \mathcal{E}$ such that $\sqrt{-1}\Theta_{h_{m}}\geq -(1/m) \omega_{X} $ after replacing $X$ with $X_{0}:=X_{\reg} \cap X_{\cal E}$.  We replace $X$ with $X_{0}=X_{\reg} \cap X_{\cal E}$.  We consider
$$
p:=\pi_{\mathcal{E}}\colon  P:=\bb{P}(\cal{E}) \lra X
\quad\text{and}\quad L:=\mathcal{O}_{\bb{P}(\cal{E}) }(1)
$$
and then apply the result of the positivity of direct images in \cite[Lemma 5.4]{CP17} (see \cite{Wan21} for the K\"ahler cases).

From the surjective morphism $p^{*} \rmS^{[m]} \mathcal{E} \to mL $ and Definition~\ref{def-psef}, we obtain singular Hermitian metrics $g_{m}$ on $L$ such that $\sqrt{-1}\Theta_{g_{m}}\geq -(1/m)p^{*}\omega_{X} $ and $\{x \,|\, \nu(g_{m}, x) >0 \}$ is not dominant over $X$ (see the proof of $(2) \Rightarrow (3)$ in Proposition~\ref{prop-main} for the details).  For a local potential $f$ with $\omega_{X} =\ddbar f$, we consider the metric $g_{m}e^{-(1/m) p^{*}f}$ on $L$ locally defined over $Y$.  Note that the curvature of $g_{m}e^{-(1/m)p^{*}f}$ is non-negative.  We apply the result of the positivity of direct images for $rL$ equipped with $(g_{m}e^{-(1/m)p^{*}f})^{r}$.  Then, the induced $L^{2}$-metric on
$$
p_{*}(K_{P/X} + rL  )= \det\, \mathcal{E}
$$ 
is positively curved and coincides with the determinant metric $\det (g_{m}e^{-(1/m)f})$.  Hence, we see that $\sqrt{-1}\Theta_{\det g_{m}}\geq -(r/m)\omega_{X} $ holds since $\det (g_{m}e^{-(1/m)f})=(\det g_{m}) \cdot e^{-(r/m)f}$.  Note that $\det g_{m}$ is a metric on $\det \mathcal{E}$ globally defined on $Y$.  This finishes the first conclusion.  The second conclusion directly follows from Proposition~\ref{prop-cor2}.
\end{proof}

\begin{proof}[Proof of Theorem~\ref{theo-main}]
Let $X$ be a projective klt variety with pseudo-effective tangent sheaf.  Then, the anti-canonical divisor $-K_{X}$ is pseudo-effective as a $\mathbb Q$-Cartier divisor by Proposition~\ref{det}.  If $K_{X}$ is pseudo-effective, then $K_{X}$ is numerically trivial; thus $X$ is a Q-abelian variety by Proposition~\ref{thm-abel}, which finishes the proof.  Hence, we may assume that $K_{X}$ is not pseudo-effective.
 
By \cite[Corollary 1.3.3]{BCHM10}, we can find a composite $\pi_{0}\colon X:=X_{0} \dashrightarrow X_{0}'$ of divisorial contractions and flips, and a Mori fiber space $f_{0}\colon X_{0}' \to X_{1} $.  The tangent sheaves of $X_{0}'$ and $X_{1}$ are pseudo-effective by Propositions~\ref{prop-bir} and~\ref{prop-fib}.  If $ X_{1}$ is one point or $K_{X}$ is pseudo-effective, then we complete the proof by using Proposition~\ref{thm-abel}; otherwise, we repeat the same argument as above for $X_{1}$.  By repeating this procedure, we obtain the conclusion.
\end{proof}

%%%%%%%%%%%%%%%%%%%%%
% References
%%%%%%%%%%%%%%%%%%%%%
\newcommand{\etalchar}[1]{$^{#1}$}

\end{document}